\documentclass{amsproc}
\usepackage{xcolor}
\usepackage{amssymb}
\newtheorem{theorem}{Theorem}[section]

\usepackage{geometry}
 \usepackage{enumitem}
\theoremstyle{definition}
\newtheorem{definition}[theorem]{Definition}

\theoremstyle{remark}

\numberwithin{equation}{section}
\usepackage{epsfig} 
\usepackage{epstopdf} 

\begin{document}

\title{On the Set-Valued Katugampola Fractional Integral: Properties and Regular Selections}



\author{Parneet Kaur}
\address{Department of Mathematics, 
Punjab Engineering college
(Deemed to be University), 
Sector 12 Chandigarh 160012, India}
\email{parneetkaur.phd24maths@pec.edu.in}


\author{Rattan Lal}
\address{Department of Mathematics, 
Punjab Engineering college
(Deemed to be University), 
Sector 12 Chandigarh 160012, India}
\email{rattanlal@pec.edu.in}

\author{Ankit Kumar}
\address{Department of Mathematics, 
Punjab Engineering college
(Deemed to be University), 
Sector 12 Chandigarh 160012, India}
\email{ankitkumar@pec.edu.in}


\date{January 1, 1994 and, in revised form, June 22, 1994.}



\begin{abstract}
In this paper, we develop a theory of generalized fractional integration for set-valued mappings for the Katugampola fractional integral, which unifies the classical Riemann–Liouville and Hadamard fractional integrals. The Katugampola fractional integral of a set-valued mapping is studied using integrable selections.  We study its properties w.r.t. the Hausdorff metric on the space of nonempty compact subsets of $\mathbb{R}$ and several fundamental analytical characteristics including convexity, boundedness and continuity are preserved under Katugampola fractional integration. Furthermore, we establish that both bounded variation and Lipschitz regularity of a set-valued mapping are preserved under its Katugampola fractional integral. We investigate the existence of regular selections associated with the Katugampola fractional integral and show that whenever the original set-valued mapping admits a selection with a specified regularity property, the corresponding Katugampola fractional integral does as well. 
\end{abstract}

\maketitle

\section{Introduction}
 Fractional calculus has emerged as a very powerful mathematical discipline that has extended the classical notions of differentiation and integration from integer to non-integer orders. This offers a natural framework for modeling processes with memory and hereditary effects. Unlike classical calculus, which operates through local and integer-order operations, fractional operators are used to encode global information about the history of a system, making them well-suited for describing complex phenomena arising in viscoelasticity, control theory, mathematical biology, and diffusion processes \cite{Kilbas,Podlubny}. 
 
Among the most foundational operators in fractional calculus, the Riemann-Liouville fractional integral is of primary interest. It was developed by the pioneering contributions of Liouville and Riemann and later formalized by Sonin and Letnikov. This fractional operator serves as the primary building block for defining fractional derivatives and solving fractional differential equations \cite{Podlubny}. Its analytical properties, including the linearity, boundedness and a semigroup composition rule, have been exploited in both theoretical and applied studies \cite{Samko}. 

The more generalized class of fractional integrals is introduced by Katugampola in \cite{Katu1}, now widely referred to as Katugampola fractional integrals. This is a generalization of the Riemann-Liouville and Hadamard fractional integrals through an additional parameter $\rho>0$. The Riemann-Liouville integral is recovered by taking $\rho=1$  and the Hadamard integral is obtained as $\rho \to 0^+$ \cite{Katu1}. More results on it are given in \cite{Matar,Priya}.  The Katugampola fractional integral has a wide application in fractal geometry, integral inequalities, and fractional differential equations.

Set-valued analysis expands the classical maps to the study of maps that assign sets, rather than single points, to each input. This provides a rigorous mathematical language for capturing uncertainty, imprecision and multi-valued dynamics arising in optimization, control theory, differential inclusions and mathematical economics \cite{Kisie, Levin, Michta1}. The foundational framework for integrating set-valued mappings was studied by Aumann \cite{Aumann} through the notion of an integral defined from the collection of all integrable selections and was further developed by Artstein \cite{Artstein} and others. Belov and Chistyakov \cite{Belov} established a fundamental selection principle for guaranteeing the existence of regular selections, that is, selections inheriting regularity properties such as bounded variation or Lipschitz continuity from set-valued mappings possessing corresponding regularity. In \cite{Michta}, selection properties of set-valued mappings and different types of set-valued integrals were studied. For a broader account of set-valued integration and selection theory, we can refer to \cite{Aubin1, Aubin2}. 

In 2015, Lupulescu \cite{Lupulescu} studied fractional calculus for interval-valued functions and established key properties, such as monotonicity, linearity, and semigroup property. Further Hermite-Hadamard-type inclusions via Riemann-Liouville fractional integrals were studied by Kara et al. \cite{Kara} and the Cauchy-type problem for interval-valued fractional differential equations was studied in \cite{Shen}. The definition of bounded variation in the Azel\'a sense is given in \cite{Clarkson}. 

In 2018, Verma \cite{Verma2}  studied the analytical properties for the singled valued Katugampola fractional integral and later for the Bivariate integral in the paper \cite{Verma}. The B. Yu et al. have explored the various properties of the singled valued Katugampola and its dimensional analysis in various settings in paper \cite{Yu,Yu2,Yu3,Yu4}. In \cite{Chandra22}, studied the box dimension of Katugampola integral of continuous function on rectangular domain. In \cite{Agrawal} the authors studied the dimension preserving approximation for the Riemann-Liouville fractional integrals.

Recently, Chandra and Abbas \cite{Chandra25} took a significant step in combining the two directions above by introducing the Riemann-Liouville fractional integral of set-valued mappings via integrable selections. Motivated by the more generalization offered by the Katugampola operator \cite{Katu1,Katu2} and by the foundational framework laid in \cite{Aubin2}, in this paper we studied the set-valued Katugampola fractional integral. We define this operator selection-wise in direct analogy with the set-valued Riemann-Liouville integral.

 This paper is organized as follows: Section \ref{Sec2} contains the preliminaries and background required for this paper. In Section \ref{sec3}, we present some basic properties of the set-valued Katugampola fractional integral such as nonemptiness, boundedness, compactness, convexity and continuity.  In Section \ref{sec4}, we prove the bounded Variation and Lipschitz continuity are preserved by the set-valued Katukampola fractional integral operator under certain conditions. In the last Section \ref{sec5}, the existence of some regular selections is shown. 
 
\section{Preliminaries} \label{Sec2}
In the present section, we give some preliminary definitions and results that are needed in the paper. Here, the notation of $I$ denotes the notation of interval $[a,b].$
\begin{definition}
    For  non-empty compact subsets $A_1$ and $A_2$, the Hausdorff distance is defined as 
    \[ H_d(A_1,A_2) =  \max\big\{\sup_{a_2\in A_2} \inf_{a_1\in A_1}|a_1-a_2|,\sup_{a_1\in A_1} \inf_{a_2\in A_2}|a_1-a_2| \big \}. \]
\end{definition}
Let $\mathrm{K}(\mathbb{R})$ denote the collection of all non-empty compact subsets of $\mathbb{R}$ and note that $(\mathrm{K}(\mathbb{R}),H_d)$ is a complete metric space.  
\begin{definition}
Let $F$ be a set-valued map. A  single valued, measurable function  $f:I \to \mathbb{R}$ is called a selection of $F$ if  $f(t) \in F(t) $ a.e. and $f \in L^1(I).$ 
\end{definition}
Let $\mathfrak{I}$ be the collection of all integrable selections and $L^1(I)$ be the collection of all Lebesgue integrable functions over $I.$ 
\begin{definition}
  Let $F:I\rightrightarrows \mathbb{R}$ be a set-valued function and $\mathfrak{I}$ be the collection of all integrable selections of $F.$ Riemann-Liouville set-valued fractional integral of $F$ is defined as 
    \begin{equation*}
        \begin{aligned}
            _a\mathfrak{D}^{\nu}F(t)=\bigg\{\frac{1}{\Gamma(\nu)}\int_a^t(t-x)^{\nu-1}f(x)dx:f\in\mathfrak{I}\bigg\}
        \end{aligned}
    \end{equation*}
    where $\nu>0$ is a real number.
\end{definition}
    \begin{definition}
        The Katugampola fractional integral of a function  $f:I \to \mathbb{R}$ is defined as 
        \begin{equation*}
            \begin{aligned}
        _a^\rho\mathfrak{D}^\nu f(t)= \frac{(\rho+1)^{1-\nu}}{\Gamma(\nu)}\int_a^t(t^{\rho+1}-x^{\rho+1})^{\nu-1}x^{\rho}f(x)dx
    \end{aligned}
        \end{equation*}
        where $\nu>0$ and $\rho \neq -1$ are real numbers. 
    \end{definition}
\begin{definition} \label{K1}
Let $F:I\rightrightarrows \mathbb{R}$ be a set-valued function and $\mathfrak{I}$ be the collection of all integrable selections of $F.$ The set-valued Katugampola fractional integral is defined as 
\begin{equation*}
    \begin{aligned}
        _a^\rho\mathfrak{D}^\nu F(t)=\bigg\{\frac{(\rho+1)^{1-\nu}}{\Gamma(\nu)}\int_a^t(t^{\rho+1}-x^{\rho+1})^{\nu-1}x^{\rho}f(x)dx:f\in\mathfrak{I}\bigg\}
    \end{aligned}
\end{equation*}
where $\nu>0$ and $\rho \neq -1$ are real numbers.
\end{definition}
    
\begin{definition}
    Let $F$ be a compact set-valued function such that $F:I\rightrightarrows \mathbb{R}$ and for each $u\in [a,b]$, $F(u)$ is compact. For any partition $\mathcal{T}=\{t_0,t_1,t_2,\ldots t_n\}\subset I= [a,b]$ and $n\in \mathbb{N},$ the total Jordan variation of $ F$ is
    \[ V_a^b(F)=\sup_{\mathcal{T}} \sum_{i=1}^nH_d(F(t_i),F(t_{i-1})), \]
    with the supremum over all partitions $\mathcal{T}$ of $I.$ If $V_a^b(F)< \infty,$ then we say that $F$ is of bounded variation on $I.$ 

 \begin{definition}
     A set valued map $F:I\to \mathbb{R}$ is said to be Lipschitz on $I$ with respect to $H_d$ if there exists a constant $C_F>0$ such that 
     \[ H_d(F(t_1),F(t_2)) \le C_F|t_1-t_2| \text{ for all } t_1,t_2 \in I \] 
     and such a constant $C_F$ is called the Lipschitz constant of $F.$ 
 \end{definition}   
    \begin{definition}
        A set-valued map $F:I \rightrightarrows\mathbb{R}$ is said to be an integrably bounded function if there exists an integrable function $m:I\to \mathbb{R}$, such that  for all $t\in I$ and $s\in F(t)$ satisfies the $|s|\le m(t).$
    \end{definition}
    
\end{definition}

\section{Some Basic Properties of the Set-Valued Katugampola Integral} \label{sec3}
\begin{theorem} \label{bounded}
Let $\mathcal{I}=[a,b]$ and $0 \leq a<b$ be a subset of $\mathbb{R}.$ Define the set-valued map $F:\mathcal{I}\rightrightarrows\mathbb{R}$ such that $F(t)$ is nonempty, convex and compact for all $t\in \mathcal{I}.$
If $F$ is integrably bounded and Borel measurable then the set-valued Katugampola integral is nonempty and bounded on $\mathcal{I}.$
\end{theorem}
 \begin{proof}
    Firstly, we show that $_a^\rho\mathfrak{D}^\nu F(t)$ is non-empty set. Since $F$ is non-empty and closed, it admits a Borel-measurable value. By measurable selection theorem in \cite{Aumann} there exists a Lebesgue measurable mapping $f:\mathcal{I}\to \mathbb{R}$ such that \[ f(t)\in F(t) \text{ for a.e. } t\in \mathcal{I},\] 
    and also integrable bounded condition implies that for all $s\in F(t)$ the $|s|\le m(t)$ where $m:\mathcal{I} \to \mathbb{R}$ is an integrable function, hence every measurable selection is integrable. Consequently, the integral \[ \int_a^t(t^{\rho+1}-x^{\rho+1})^{\nu-1}x^{\rho}f(x)dx\] is well defined. Defining \[ \frac{(\rho+1)^{1-\nu}}{\Gamma(\nu)}\int_a^t(t^{\rho+1}-x^{\rho+1})^{\nu-1}x^{\rho}f(x)dx,\] this quantity belongs to $_a^\rho\mathfrak{D}^\nu F(t).$ Hence, $_a^\rho\mathfrak{D}^\nu F(t)$ is a non-empty set. \\ Next, we show that $_a^\rho\mathfrak{D}^\nu F(t)$ is bounded. As $F$ is bounded on $\mathcal{I}$, therefore, it will satisfy    $\sup_{t\in I}H_d(F(t),\{0\})<\infty.$ This can be simplified as,  
    \begin{equation*}
        \begin{aligned}
     H_d(F(t),\{0\})=\max\bigg\{ \sup_{s\in F(t)}|s|,\inf_{s \in F(t) }|s|\bigg\}=\sup_{s \in F(t)}|s|,
        \end{aligned}
    \end{equation*}
    hence, $F$ is bounded it satisfies the $ M= \sup_{t\in I}\sup_{s \in F(t) }|s| < \infty.$
    \begin{equation*}
        \begin{aligned}
            H_d(_a^\rho\mathfrak{D}^\nu F(t),\{0\})=&\sup_{ y \in _a^\rho\mathfrak{D}^\nu F(t)} |y|\\
            =& \sup_{f \in \mathfrak{I}}\Bigg|\frac{(\rho+1)^{1-\nu}}{\Gamma(\nu)}\int_a^t(t^{\rho+1}-x^{\rho+1})^{\nu-1}x^{\rho}f(x)dx\Bigg|\\ 
            \le&  \frac{(\rho+1)^{1-\nu}}{\Gamma(\nu)} \sup_{f \in \mathfrak{I}}\int_a^t |(t^{\rho+1}-x^{\rho+1})^{\nu-1}x^{\rho}||f(x)|dx 
              \end{aligned}
    \end{equation*}
    For finding the bound 
    \begin{equation*}
        \begin{aligned}
            \sup_{t \in I} H_d(_a^\rho\mathfrak{D}^\nu F(t),\{0\})\le& \frac{(\rho+1)^{1-\nu}}{\Gamma(\nu)}  \sup_{t \in I}\sup_{f \in \mathfrak{I}}\int_a^t |(t^{\rho+1}-x^{\rho+1})^{\nu-1}x^{\rho}||f(x)|dx \\
            \le&M\frac{(\rho+1)^{1-\nu}}{\Gamma(\nu)}\int_a^t |(t^{\rho+1}-x^{\rho+1})^{\nu-1}x^{\rho}|dx\\
            \le& M\frac{(\rho+1)^{1-\nu}(\rho+1)^{-1}}{\nu\Gamma(\nu)} (t^{\rho+1}-x^{\rho+1})^{\nu}\Big|_a^t\\
            =& M\frac{(\rho+1)^{-\nu}}{\Gamma(\nu+1)} (t^{\rho+1}-a^{\rho+1})^{\nu}
        \end{aligned}
    \end{equation*}
    From this we can conclude that \[\sup_{t \in I} H_d(_a^\rho\mathfrak{D}^\nu F(t),\{0\})\le M\frac{(\rho+1)^{-\nu}}{\Gamma(\nu+1)} (b^{\rho+1}-a^{\rho+1})^{\nu}.\] 
    With this the boundedness of Katugampola integral is proved.
 \end{proof}
 Throught the paper, we denote the notation of $\mathcal{I}=[a,b]$ where $0\leq a<b.$
 \begin{theorem} \label{convex}
     Let the set-valued map $F:\mathcal{I}\to\mathbb{R}$ be such that $F(x)$ is convex for every $x\in \mathcal{I}$. Then for every $\nu > 0$ and $ \rho \neq -1 $ the Katugampola fractional integral of F preserves convexity, that is, $_a^\rho\mathfrak{D}^\nu F(t)$ is convex for all $t\in \mathcal{I}.$
 \end{theorem}
 \begin{proof}
 Let $t \in \mathcal{I},$ if $_a^\rho\mathfrak{D}^\nu F(t)$ is empty, we are done. If not, then for any $ s_1, s_2 \in _a^\rho\mathfrak{D}^\nu F(t),$ then there exist integrable selections $f_1,f_2$ such that 
 \[ s_i=\frac{(\rho+1)^{1-\nu}}{\Gamma(\nu)}\int_a^t(t^{\rho+1}-x^{\rho+1})^{\nu-1}x^{\rho}f_i(x)dx,~~i=1,2\]  
 Let $ \lambda \in [0,1].$ And consider the convex combination $\lambda s_1+(1-\lambda)s_2.$ From this we get the
 \begin{equation*}
     \begin{aligned}
  \lambda s_1&+(1-\lambda)s_2\\=&\lambda\frac{(\rho+1)^{1-\nu}}{\Gamma(\nu)}\int_a^t(t^{\rho+1}-x^{\rho+1})^{\nu-1}x^{\rho}f_1(x)dx+(1-\lambda)\frac{(\rho+1)^{1-\nu}}{\Gamma(\nu)}\int_a^t(t^{\rho+1}-x^{\rho+1})^{\nu-1}x^{\rho}f_2(x)dx
  \end{aligned}
 \end{equation*}
 Since $F(t)$ is a convex for each $t,$ therefore $f(t)=\lambda f_1(t)+(1-\lambda)f_2(t)$ is also an integral selection for $F(t).$ Therefore, by linearity of the Lebesgue integral
 \begin{equation*}
     \begin{aligned}
    \lambda s_1+(1-\lambda)s_2=\frac{(\rho+1)^{1-\nu}}{\Gamma(\nu)}\int_a^t(t^{\rho+1}-x^{\rho+1})^{\nu-1}x^{\rho}f(x)dx \in _a^\rho\mathfrak{D}^\nu F(t)
     \end{aligned}
 \end{equation*}
 Hence, every convex combination of elements of $_a^\rho\mathfrak{D}^\nu F(t)$ remains in the set. This proves that $_a^\rho\mathfrak{D}^\nu F(t)$ is convex.
 \end{proof}
 \begin{theorem} \label{compact,continuity}
    Let the set valued mapping $F:\mathcal{I} \rightrightarrows \mathbb{R}$ with non-empty compact values, Borel measurable and integrably bounded on $\mathcal{I}.$ Then, for each $t\in \mathcal{I},$ the set valued Katugampola integral is non empty and compact and the mapping  to $_a^\rho\mathfrak{D}^\nu F(t)$ preserves the continuity on $\mathcal{I}$, for $\rho \neq -1, \nu >0,$ in the Hausdorff metric $H_d$. 
\end{theorem}
\begin{proof}
     By theorem \ref{bounded} for each $t\in \mathcal{I},$ the set $_a^\rho\mathfrak{D}^\nu F(t)$ is non-empty. Let $\mathfrak{I}$ be the collection of integrable selections of F. By integrably boundedness, there exists $m_1:\mathcal{I}\to \mathbb{R}$ such that $\sup\{ |x|:x\in F(u)\}\le m_1(t)$ for a.e. $t\in \mathcal{I}.$ For each integrable selection $f,$ we have $|f(t)|\le m_1(t)$ for a.e. $t \in \mathcal{I}.$ From the definition \ref{K1} the set-valued Katugampola is 
     \begin{equation*}
    \begin{aligned}
        _a^\rho\mathfrak{D}^\nu F(t)=\bigg\{\frac{(\rho+1)^{1-\nu}}{\Gamma(\nu)}\int_a^t(t^{\rho+1}-x^{\rho+1})^{\nu-1}x^{\rho}f(x)dx:f\in\mathfrak{I}\bigg\}
    \end{aligned}
\end{equation*}
For any $x\in \mathcal{I}$ we can have the 
\begin{equation*}
    \begin{aligned}
      \Bigg|  \frac{(\rho+1)^{1-\nu}}{\Gamma(\nu)}\int_a^t(t^{\rho+1}-x^{\rho+1})^{\nu-1}x^{\rho}f(x)dx \Bigg| &\le \frac{(\rho+1)^{1-\nu}}{\Gamma(\nu)}\int_a^t(t^{\rho+1}-x^{\rho+1})^{\nu-1}x^{\rho} |f(x)|dx \\ &\le \frac{(\rho+1)^{1-\nu}}{\Gamma(\nu)}\int_a^t(t^{\rho+1}-x^{\rho+1})^{\nu-1}x^{\rho} m_1(x)dx
    \end{aligned}
\end{equation*}
Since we have the condition of $\nu>0, ~~\rho\neq-1$ and $m \in L^1(\mathcal{I}).$ Therefore there exist $K>0$ such that it bounds it is for all $t \in \mathcal{I}$ and for all $f\in \mathfrak{I}.$ Thus, $_a^\rho\mathfrak{D}^\nu F(t)$ is bounded for all $t\in \mathcal{I}.$ Now, to check the closure of  $_a^\rho\mathfrak{D}^\nu F(t).$ Let $\{x_n\} \in _a^\rho\mathfrak{D}^\nu F(t)$ is a converging sequence to some $x\in \mathbb{R}.$ Then there exist a integrable selection $f_n$ of F such that 
\[ x_n= \frac{(\rho+1)^{1-\nu}}{\Gamma(\nu)}\int_a^t(t^{\rho+1}-x^{\rho+1})^{\nu-1}x^{\rho}f_n(x)dx. \] 
As $F$ is integrably bounded, therefore each of $f_n$ is integrably bounded, such that $|f_n(t)|\le m_2(t)$ for a.e. $t\in \mathcal{I}$ and $n\in \mathbb{N}.$ By the Dunford-Pettis Theorem, there exist a subsequence $g_{m}$ and a function $g \in L^1$ such that
\[
g_{m} \rightharpoonup g
\quad \text{weakly in } L^1.
\]
Hence,
\begin{equation*}
    \begin{aligned}
 \int_a^t(t^{\rho+1}-x^{\rho+1})^{\nu-1}x^{\rho}g_m(x)dx &\to \int_a^t(t^{\rho+1}-x^{\rho+1})^{\nu-1}x^{\rho}g(x)dx \\
 \frac{(\rho+1)^{1-\nu}}{\Gamma(\nu)}\int_a^t(t^{\rho+1}-x^{\rho+1})^{\nu-1}x^{\rho}g_m(x)dx & \to \frac{(\rho+1)^{1-\nu}}{\Gamma(\nu)}\int_a^t(t^{\rho+1}-x^{\rho+1})^{\nu-1}x^{\rho}g(x)dx =y(t)
  \end{aligned}
\end{equation*}
Since $y(t) \in ~_a^\rho\mathfrak{D}^\nu F(t)$ for every $t\in \mathcal{I}$. Therefore, $_a^\rho\mathfrak{D}^\nu F(t)$ is closed and together with the boundedness is compact. Now we show the continuity of $_a^\rho\mathfrak{D}^\nu F.$ Let $t,p \in \mathcal{I}$ and $f \in \mathfrak{I}.$ Then  
\begin{equation*}
    \begin{aligned}
\frac{(\rho+1)^{1-\nu}}{\Gamma(\nu)}\int_a^t(t^{\rho+1}&-x^{\rho+1})^{\nu-1}x^{\rho}f(x)dx-\frac{(\rho+1)^{1-\nu}}{\Gamma(\nu)}\int_a^p(p^{\rho+1}-x^{\rho+1})^{\nu-1}x^{\rho}f(x)dx \\ &=\frac{(\rho+1)^{1-\nu}}{\Gamma(\nu)}\bigg( \int_a^p\big((t^{\rho+1}-x^{\rho+1})^{\nu-1}-(p^{\rho+1}-x^{\rho+1})^{\nu-1}\big)x^{\rho}f(x)dx \\ & \hspace{15em} + \int_p^t(t^{\rho+1}-x^{\rho+1})^{\nu-1}x^{\rho}f(x)dx\bigg)
   \end{aligned}
\end{equation*}
By the boundedness of $f(t)\le m_1(t).$ Now,
\begin{equation*}
    \begin{aligned}
        \bigg| \frac{(\rho+1)^{1-\nu}}{\Gamma(\nu)}\bigg( \int_a^p\big((t^{\rho+1}-x^{\rho+1})^{\nu-1}-(p^{\rho+1}-x^{\rho+1})^{\nu-1}\big)x^{\rho}f(x)dx  + \int_p^t(t^{\rho+1}-x^{\rho+1})^{\nu-1}x^{\rho}f(x)dx \bigg| \\ \le \frac{(\rho+1)^{1-\nu}}{\Gamma(\nu)} \int_a^p\big|(t^{\rho+1}-x^{\rho+1})^{\nu-1}-(p^{\rho+1}-x^{\rho+1})^{\nu-1}\big|x^{\rho}f(x)dx\\  + \frac{(\rho+1)^{1-\nu}}{\Gamma(\nu)}\int_p^t(t^{\rho+1}-x^{\rho+1})^{\nu-1}x^{\rho}f(x)dx 
    \end{aligned}
\end{equation*}
Let the term on right in the above equation be denoted by $ Z(t,p).$ We now show that $Z(t,p)\to 0$ as $t\to p.$ Since $p \in \mathcal{I},$ for a.e. $x \in [a,p)$ 
\[(t^{\rho+1}-x^{\rho+1})^{\nu-1} \to (p^{\rho+1}-x^{\rho+1})^{\nu-1} \text{ as } t  \to p \]
and as $ \rho \neq -1 $ and $ \nu > 0$ then integral 
\begin{equation*}
    \begin{aligned}
 \int_a^p\big|(t^{\rho+1}-x^{\rho+1})^{\nu-1}-&(p^{\rho+1}-x^{\rho+1})^{\nu-1}\big|x^{\rho}f(x)dx \\ & \le \int_a^p\big((t^{\rho+1}-x^{\rho+1})^{\nu-1}+(p^{\rho+1}-x^{\rho+1})^{\nu-1}\big)x^{\rho } m_1(x)dx  
  \end{aligned}
\end{equation*}
and since $ \big((t^{\rho+1}-x^{\rho+1})^{\nu-1}+(p^{\rho+1}-x^{\rho+1})^{\nu-1}\big) \in L^1(\mathcal{I})$ and therefore we can conclude Lebesgue's dominated convergence theorem 
\begin{equation*}
    \begin{aligned}
 \int_a^p\big|(t^{\rho+1}-x^{\rho+1})^{\nu-1}-&(p^{\rho+1}-x^{\rho+1})^{\nu-1}\big|x^{\rho}m_1(x)dx \to 0 \text{ as } t\to p.
  \end{aligned}
\end{equation*}
Now, for the other part of the integral 
\begin{equation*}
    \begin{aligned}
       0\le \int_p^t(t^{\rho+1}-x^{\rho+1})^{\nu-1}x^{\rho}f(x)dx \le  \int_p^t(t^{\rho+1}-x^{\rho+1})^{\nu-1}x^{\rho}m_1(x) dx \to 0 \text{ as } t\to p 
    \end{aligned}
\end{equation*}
because the integration converges to the point which is $0.$ Take any point $u\in _a^\rho\mathfrak{D}^\nu F(t),$ such that there is some $f \in \mathfrak{I}$ for which the 
\begin{equation*}
    \begin{aligned}
u= \frac{(\rho+1)^{1-\nu}}{\Gamma(\nu)}\int_a^t(t^{\rho+1}&-x^{\rho+1})^{\nu-1}x^{\rho}f(x)dx
 \end{aligned}
\end{equation*}
Take $v\in _a^\rho\mathfrak{D}^\nu F(p)$ such that its integral is defined for the same $f\in \mathfrak{I}$ as that taken for the point $u.$ By the inequality proved in above $ |u-v| \le Z(t,p),$ and hence the 
      \begin{equation*}
          \begin{aligned}
              \sup_{u\in  _a^\rho\mathfrak{D}^\nu F(t)} \inf_{v\in _a^\rho\mathfrak{D}^\nu F(p)}|u-v| \le Z(t,p)
          \end{aligned}
      \end{equation*}
By the definition of Hausdorff metric, we have 
\[ H_d(_a^\rho\mathfrak{D}^\nu F(p),_a^\rho\mathfrak{D}^\nu F(t)) \le Z(p,t) \] for all $t,p \in \mathcal{I}.$ Since, $Z(t,p)\to 0$ as the $t\to p.$ This shows the mapping $ _a^\rho\mathfrak{D}^\nu F(t)$ on $\mathcal{I}$ is continuous with the metric $H_d.$ This completes the proof. 
\end{proof}
\section{Bounded Variation and Lipschitz Continuity of the  Set-Valued Katugampola Integral} \label{sec4}
 \begin{theorem} \label{BoundVar}
     Let the set-valued mapping $F:\mathcal{I} \rightrightarrows \mathbb{R} $ with non empty compact convex values, Borel measurable and integrably bounded on $\mathcal{I}.$ Assume that $F$ is of bounded variation on $\mathcal{I}$ with respect to $H_d$ and that $\rho \neq -1, \nu >1$ Then, Katugampola integral $t\in \mathcal{I},$ the set $_a^\rho\mathfrak{D}^\nu F(t)$ is nonempty, compact and convex, and the mapping  to $_a^\rho\mathfrak{D}^\nu F(t)$ preserves the bounded variation on $\mathcal{I}$ in Hausdorff metric, $H_d.$
 \end{theorem}
 \begin{proof}
     Since $F(t)$ is non-empty, compact, Borel-measurable and integrably bounded, by Theorem \ref{compact,continuity}, for each $t\in \mathcal{I},$ the set $_a^\rho\mathfrak{D}^\nu F(t)$ is non-empty, compact and the mapping $_a^\rho\mathfrak{D}^\nu F(t)$ is continuous on $\mathcal{I}$ with respect to $H_d.$ Furthermore, since $F(t)$ is convex,Theorem \ref{convex} gives us that the mapping $_a^\rho\mathfrak{D}^\nu F(t)$ is convex. Thus, it is sufficient to verify that the mapping $_a^\rho\mathfrak{D}^\nu F(t)$ is of bounded variation. Also, for each $t\in \mathcal{I},$ the set $ F(t)$ nonempty, compact and convex subset of $\mathbb{R}$ and hence it must be a closed interval. 
     \[ \underline{f}(t):= \inf F(t)  \text{ and } \overline{f}(t):=\sup F(t) \]
    The set valued $F(t)$ can be written as 
    \[ F(t)= [\underline{f}(t),\overline{f}(t)], \text{ for all } t\in \mathcal{I} \] 
  Consider a partition $\Gamma=\{0\le a=t_1 <t_2 <\ldots < t_n=b \}$ of interval $\mathcal{I}.$ Now, the Hausdorff distance is measured as the 
  \begin{equation*}
      \begin{aligned}
  H_d(F(t_i),F(t_{i-1}))=& H_d([\underline{f}(t_i),\overline{f}(t_i)],[\underline{f}(t_{i-1}),\overline{f}(t_{i-1})])\\
  =&\max \big\{ |\underline{f}(t_i)-\underline{f}(t_{i-1})|, |\overline{f}(t_{i})-\overline{f}(t_{i-1})|\big\}.
   \end{aligned}
  \end{equation*}
  From this 
   \begin{equation*}
       \begin{aligned}
           \sum_{i=1}^{n}|\overline{f}(t_{i})-\overline{f}(t_{i-1})|\le& \sum_{i=1}^nH_d(F(t_i),F(t_{i-1}))  \text{ and also, } \\
           \sum_{i=1}^{n}|\underline{f}(t_i)-\underline{f}(t_{i-1})|\le& \sum_{i=1}^nH_d(F(t_i),F(t_{i-1})).
       \end{aligned}
   \end{equation*}
  If we take the supremum on all partitions $\Gamma,$ we get
  \begin{equation*}
      \begin{aligned}
          V_a^b(\underline{f})\le V_a^b(F)< \infty  \text{ and }  V_a^b(\overline{f})\le V_a^b(F) < \infty
      \end{aligned}
  \end{equation*}
as $F$ is of bounded variation. From the bounded variation of functions $\underline{f}$ and $ \overline{f}$ on $I,$ these are also bounded on $\mathcal{I}.$ For $t\in I$ define 
\begin{equation*}
    \begin{aligned}
        (_a^\rho\mathfrak{D}^\nu \underline{f})(t)=\frac{(\rho+1)^{1-\nu}}{\Gamma(\nu)}\int_a^t(t^{\rho+1}-x^{\rho+1})^{\nu-1}x^{\rho}\underline{f}(x)dx.
    \end{aligned}
\end{equation*}
Similarly, 
\begin{equation*}
    \begin{aligned}
        (_a^\rho\mathfrak{D}^\nu \overline{f})(t)=\frac{(\rho+1)^{1-\nu}}{\Gamma(\nu)}\int_a^t(t^{\rho+1}-x^{\rho+1})^{\nu-1}x^{\rho}\overline{f}(x)dx.
    \end{aligned}
\end{equation*}
Because $\underline{f}$ and $\overline{f}$ are both bounded and integrable, thus the both operators $_a^\rho\mathfrak{D}^\nu \underline{f}$ and $_a^\rho\mathfrak{D}^\nu \overline{f}$ are well defined. We now assert that each of these operator is Lipschitz and therefore, is of bounded variation on $\mathcal{I}$. Now, as the $\underline{f}$ is bounded, there exist a $M>0$ such that $\underline{f}(t)\le M$ for all $t\in \mathcal{I}. $ Also, as for all $t \in (a,b]$ the operator $_a^\rho\mathfrak{D}^\nu(t)$ is differentiable as the following 
\begin{equation*}
    \begin{aligned}
        (_a^\rho\mathfrak{D}^\nu)' \underline{f}(t)=\frac{(\rho+1)^{2-\nu}(\nu-1)}{\Gamma(\nu)}\int_a^tt^\rho(t^{\rho+1}-x^{\rho+1})^{\nu-2}x^{\rho}\underline{f}(x)dx.
    \end{aligned}
\end{equation*}
Hence, 
\begin{equation*}
    \begin{aligned}
        |(_a^\rho\mathfrak{D}^\nu)' \underline{f}(t)|&=\frac{(\rho+1)^{2-\nu}(\nu-1)}{\Gamma(\nu)}\int_a^tt^\rho(t^{\rho+1}-x^{\rho+1})^{\nu-2}x^{\rho}|\underline{f}(x)|dx \\
        &\le \frac{(\rho+1)^{2-\nu}(\nu-1)}{\Gamma(\nu)}M\int_a^tt^\rho(t^{\rho+1}-x^{\rho+1})^{\nu-2}x^{\rho}dx 
    \end{aligned}
\end{equation*}
Put $t^{\rho+1}-x^{\rho+1}=u$ then $x^\rho dx=-\frac{du}{\rho+1}$ and now the $\int u^{\nu-2}du=\frac{u^{\nu-1}}{\nu-1}.$ So, we get 
\begin{equation*}
    \begin{aligned}
        |(_a^\rho\mathfrak{D}^\nu)' \underline{f}(t)|&\le  \frac{M(\rho+1)^{2-\nu}(\nu-1)}{\Gamma(\nu)} \frac{t^\rho(t^{\rho+1}-a
        ^{\rho+1})^{(\nu-1)}}{(\rho+1)(\nu-1)}\\
        &\le \frac{M(\rho+1)^{1-\nu}}{\Gamma(\nu)} b^\rho (b^{\rho+1}-a
        ^{\rho+1})^{(\nu-1)}
    \end{aligned}
\end{equation*}
Hence, the $(_a^\rho\mathfrak{D}^\nu)' \underline{f}(t)$ is bounded on interval $(a,b].$ Moreover, as $ t \to a ^+$ can be continuously extends to $[a,b],$ so derivative is bounded on entire interval $\mathcal{I}.$ So the $(_a^\rho\mathfrak{D}^\nu)\underline{f}(t)$ is Lipschitz on $\mathcal{I}$ with Lipschitz constant 
\[ \frac{M(\rho+1)^{-\nu}}{\Gamma(\nu)} b^\rho (b^{\rho+1}-a
        ^{\rho+1})^{(\nu-1)} \]
        Similarly, we can show for the $(_a^\rho\mathfrak{D}^\nu) \overline{f}(t)$ is also Lipschitz with same Lipschitz constant. Thus, both are functions of bounded variation of $\mathcal{I}.$\\
        We now established that for every $t\in \mathcal{I}$
    \begin{equation} \label{interval}
        \begin{aligned}  
     _a^\rho\mathfrak{D}^\nu F(t)=[(_a^\rho\mathfrak{D}^\nu) \underline{f}(t),(_a^\rho\mathfrak{D}^\nu)\overline{f}(t) ]. 
       \end{aligned}
    \end{equation}  
        Let $f$ is a integrable selection of the set-valued mappings $F.$ Then, for each $t\in \mathcal{I},$ it satisfies 
        \[ \underline{f}(t)\le f(t) \le \overline{f}(t).\] 
        Since the kernel \[\frac{(\rho+1)^{1-\nu}(t^{\rho+1}-x^{\rho+1})^{\nu-1}x^{\rho}}{\Gamma(\nu)}\] is non negative for all $x\in [a,t],$ integrating the above inequality yields $ (_a^\rho\mathfrak{D}^\nu) \underline{f}(t) \le _a^\rho\mathfrak{D}^\nu f(t)\le  (_a^\rho\mathfrak{D}^\nu) \overline{f}(t).$ 
        This shows that every element of $ _a^\rho\mathfrak{D}^\nu F(t) $ lies within the  interval of (\ref{interval}). Clearly, both $\underline{f}$ and $\overline{f}$ are selections of $F.$ Therefore, their corresponding fractional integral $(_a^\rho\mathfrak{D}^\nu \underline{f})(t)$ and $(_a^\rho\mathfrak{D}^\nu \overline{f})(t)$ belong to $_a^\rho\mathfrak{D}^\nu F(t).$ Hence the interval in \ref{interval} is contained in $_a^\rho\mathfrak{D}^\nu F(t).$ Combining both the inclusions gives the equality of (\ref{interval}).\\
      Now consider a partition $P=\{a=t_0<t_1<\ldots <t_n=b \}$ of the interval $\mathcal{I}.$ Then for each of $i=1,2\ldots n$ we have 
      \begin{equation*}
          \begin{aligned}
       H_d(_a^\rho\mathfrak{D}^\nu F(t_i),_a^\rho\mathfrak{D}^\nu F(t_{i-1}))=& H_d([_a^\rho\mathfrak{D}^\nu \underline{f}(t_i),_a^\rho\mathfrak{D}^\nu \overline{f}(t_i)],[_a^\rho\mathfrak{D}^\nu \underline{f}(t_{i-1}),_a^\rho\mathfrak{D}^\nu \overline{f}(t_{i-1})]) \\
       =& \max\{|_a^\rho\mathfrak{D}^\nu \underline{f}(t_i)-_a^\rho\mathfrak{D}^\nu \underline{f}(t_{i-1})|,|_a^\rho\mathfrak{D}^\nu \overline{f}(t_i)-_a^\rho\mathfrak{D}^\nu \overline{f}(t_{i-1})| \}
        \end{aligned}
      \end{equation*} 
      Hence, 
      \[ \sum_{i=1}^n H_d(_a^\rho\mathfrak{D}^\nu F(t_i),_a^\rho\mathfrak{D}^\nu F(t_{i-1})) \le \sum_{i=1}^n (|_a^\rho\mathfrak{D}^\nu \underline{f}(t_i)-_a^\rho\mathfrak{D}^\nu \underline{f}(t_{i-1})|+|_a^\rho\mathfrak{D}^\nu \overline{f}(t_i)-_a^\rho\mathfrak{D}^\nu \overline{f}(t_{i-1})| ) \] 
      by taking the supremum over all the partitions $P$ of $\mathcal{I}=[a,b].$ We get 
      \[ V_a^b(_a^\rho\mathfrak{D}^\nu F ) \le V_a^b(_a^\rho\mathfrak{D}^\nu \underline{f}) +V_a^b(_a^\rho\mathfrak{D}^\nu \overline{f})< \infty,\] 
      because both $_a^\rho\mathfrak{D}^\nu \underline{f}$ and $_a^\rho\mathfrak{D}^\nu \overline{f}$ are functions of bounded variations. Therefore, the mapping of the $_a^\rho\mathfrak{D}^\nu F$ is of bounded variation on $\mathcal{I}$ with respect to the Hausdorff metric $H_d.$ This completes the argument.  
 \end{proof}
 \begin{theorem}
   Let the set-valued mapping $F:\mathcal{I} \rightrightarrows \mathbb{R} $ with non empty compact convex values, Borel measurable and integrably bounded on $\mathcal{I}.$ Assume that $F$ is Lipschitz on $\mathcal{I}$ with respect to $H_d$ and that $\rho \neq -1, \nu >1.$ Then, for every $t\in \mathcal{I},$ the Katugampola integral $_a^\rho\mathfrak{D}^\nu F(t)$ is nonempty, compact and convex, and the mapping to $_a^\rho\mathfrak{D}^\nu F(t)$ preserves the Lipschitz on $\mathcal{I}$ in Hausdorff metric, $H_d.$
     \end{theorem}
\begin{proof}
   We need to show that $_a^\rho\mathfrak{D}^\nu F(t)$ is  Lipschitz on $\mathcal{I}$ as from the Theorem \ref{compact,continuity}, for each $t\in \mathcal{I}$ the set $_a^\rho\mathfrak{D}^\nu F(t)$ is nonempty and compact. As $F$ is Lipschitz on $\mathcal{I}$, let $C_F$ is Lipschitz constant of F, 
   \[ H_d(F(t_1),F(t_2))\le C_F |t_1-t_2| \text{ for all } t_1,t_2 \in \mathcal{I}.\]  
   Let's define the uniform bound on $F,$ for the value of $t_0 \in \mathcal{I}.$ Since $F(t_0)$ is compact, therefore bounded. Define  
\[ B_0:= \sup \{|u| : u \in F(t_0)\} < \infty \] 
Now consider any $t\in I$ and any $u \in F(t).$ By the definition of the Hausdorff distance, there exist some $ v \in F(t_0)$ such that 
\[ |u-v|\le H_d(F(t),F(t_0)).\]
Using the Lipschitz property of $F,$ we obtain 
\[ |u-v|\le H_d(F(t),F(t_0)) \le C_F|t-t_0|\le C_F|b-a|.\]
And we have $|v|\le B_0$ as $v\in F(t_0).$ Therefore
\[ |u| \le |v|+|u-v|\le B_0+C_F(b-a) \]
Let \[ A:= B_0+C_F(b-a). \]
Then it follows that
\[ \sup \{|u|:u \in F(t)\}\le A \text{ for all } t\in \mathcal{I}.\]
As a consequence, any integrable selection $f$ of $F$ satisfies
\begin{equation} \label{f_bound}
    \begin{aligned}
 |f(t)| \le A  \text{ for all } t \in \mathcal{I}. 
  \end{aligned}
\end{equation}
 For such a selection of $f,$ define a function for $t\in \mathcal{I},$ 
 \[ I_f(t):= \frac{(\rho+1)^{1-\nu}}{\Gamma(\nu)}\int_a^t(t^{\rho+1}-x^{\rho+1})^{\nu-1}x^{\rho}f(x)dx \]
 Since $\rho \neq -1, \nu>1$ and the selection of $f$ is bounded, we differentiation under the integral. Thus, for every $t \in (a,b],$ we obtain 
 \[I_f'(t)=\frac{(\rho+1)^{2-\nu}(\nu-1)}{\Gamma(\nu)}\int_a^t t^\rho (t^{\rho+1}-x^{\rho+1})^{\nu-2}x^{\rho}f(x)dx  \]
 Using the estimate in \ref{f_bound}, it follows that
 \begin{equation*}
     \begin{aligned}
         |I_f'(t)| &\le \frac{(\rho+1)^{2-\nu}(\nu-1)}{\Gamma(\nu)}A\int_a^t t^\rho (t^{\rho+1}-x^{\rho+1})^{\nu-2}x^{\rho}dx\\ &= \frac{(\rho+1)^{2-\nu}(\nu-1)}{\Gamma(\nu)}A\bigg[\frac{t^\rho (t^{\rho+1}-a^{\rho+1})^{\nu-1}}{(\rho+1)(\nu-1)} \bigg] \\ 
         &= \frac{A(\rho+1)^{1-\nu}}{\Gamma(\nu)}t^\rho (t^{\rho+1}-a^{\rho+1})^{\nu-1}\\
         &\le \frac{A(\rho+1)^{1-\nu}}{\Gamma(\nu)} b^\rho (b^{\rho+1}-a^{\rho+1})^{\nu-1}
     \end{aligned}
 \end{equation*}
 Therefore, the function $I_f$ is Lipschitz continious on $\mathcal{I},$ with Lipschitz constant $K$ is given by the 
 \[ K:=\frac{A(\rho+1)^{1-\nu}}{\Gamma(\nu)} b^\rho (b^{\rho+1}-a^{\rho+1})^{\nu-1} \] 
 Clearly, this constant does not depend on the particular choice of the selection $f.$ Now take arbitrary point $t_i,t_{i-1}\in \mathcal{I}$ and let $f$ be an integrable selection of $F.$ By the Lipschitz continuity of $I_f$ 
 \[ |I_f(t_i)-I_f(t_{i-1})|\le K|t_i-t_{i-1
 }|\] 
 Choose $u'\in _a^\rho\mathfrak{D}^\nu F(t_i),$ then there exist a selection $f$ such that $u'=I_f(t_i).$ Further, let $v' \in _a^\rho\mathfrak{D}^\nu F(t_{i-1})$  such that $v'=I_f(t_{i-1})$ and consequently 
 \[ |u'-v'| \le K|t_i-t_{i-1}|.\]
 Taking the infimum over all $v'\in _a^\rho\mathfrak{D}^\nu F(t_{i-1})$ and then the supremum over $u'\in _a^\rho\mathfrak{D}^\nu F(t_i)$ gives 
 \[ \sup_{u'\in _a^\rho\mathfrak{D}^\nu F(t_i)} \inf_{v'\in _a^\rho\mathfrak{D}^\nu F(t_{i-1})} |u'-v'| \le K|t_i-t_{i-1}| \]
 Repeating the same argument with $t_i$and $t_{i-1}$ interchanged gives 
 \[ \sup_{v'\in _a^\rho\mathfrak{D}^\nu F(t_{i-1})} \inf_{u'\in _a^\rho\mathfrak{D}^\nu F(t_i)} |u'-v'| \le K|t_i-t_{i-1}| \]
 Combining these two inequalities and using the definition of the Hausdorff metric, we conclude that
 \[ H_d(_a^\rho\mathfrak{D}^\nu F(t_i),_a^\rho\mathfrak{D}^\nu F(t_{i-1}))\le K|t_i-t_{i-1}|.\]
 Thus, the mapping $_a^\rho\mathfrak{D}^\nu F(t)$ is Lipschitz continuous on $\mathcal{I}$ with the Hausdorff distance. This completes the argument. 
 
\end{proof} 
 
\section{On Existence of Regular Selection} \label{sec5}

\begin{theorem}\cite[Theorem 3]{Belov} \label{selec} (Existence of Regular Selection) Let $(X,d)$ be a metric space. If $F:I \rightrightarrows X$ is set-valued mappings which is compact and of bounded variation. If $t \in I$ and $x\in F(t),$ then the $F$ admits the regular selection $f\in \mathcal{F},$ where $\mathcal{F}$ denotes the collection of all mappings from $[a,b]$ to $X.$ And $f$ is a bounded variation mappings such that $f(t)=x$ and $V_a^b(f)\le V_a^f(F).$
\end{theorem}

\begin{theorem} \label{Bvar}
 Let the set valued mappings $F:\mathcal{I} \rightrightarrows \mathbb{R}$ with non empty compact convex values, Borel measurable and integrably bounded on $\mathcal{I}.$ Let set valued Katugampola integral $_a^\rho\mathfrak{D}^\nu F$ of $F$ with $\rho \neq -1, \nu >1.$ Suppose $F$ is bounded variation on the interval $\mathcal{I}$ with respect to the Hausdorff metric $H_d,$ then the operator maps $t \to _a^\rho\mathfrak{D}^\nu F(t)$ guaranties the existence of a continuous selection that is also of bounded variation.
\end{theorem}
\begin{proof}
  Let  $\rho \neq -1, \nu>1$ and $F:\mathcal{I} \rightrightarrows \mathbb{R}$ with non empty compact convex values, Borel measurable and integrably bounded on $\mathcal{I}.$ Define the set valued mapping \[ H(t) := _a^\rho\mathfrak{D}^\nu F(t) \text{ for all } t \in \mathcal{I}.\]
  By applying Theorem \ref{bounded}, it follows that, for every $t \in \mathcal{I},$ the set valued Katugampola fractional integral set is non-empty and by Theorem \ref{compact,continuity}, the set $H(t)$ is compact. Further, Theorems \ref{compact,continuity} and \ref{BoundVar} imply that the set-valued $H$ is continuous on $\mathcal{I}$ and possesses bounded variation with respect to the Hausdorff metric $H_d$.
  Therefore, all assumptions required in Theorem \ref{selec} of Belov and Chistyakov \cite{Belov} are fulfilled. Consequently, there exists a regular selection
  \[ h:\mathcal{I} \to \mathbb{R}\text{ such that } h(t) \in H(t)= _a^\rho\mathfrak{D}^\nu F(t) \text{ for all } t \in \mathcal{I}.\]
  Moreover, the selected function $h$ is continuous and has bounded variation on $\mathcal{I}$. In addition, its total variation is \[ V_a^b(h)\le V_a^b(H). \]
  Hence, a continuous regular selection of the set-valued fractional integral mapping is obtained.
  \end{proof}
  \begin{theorem}(Existence of more regular selections) \label{moreselecton} \cite[Theorem 3]{Belov}
      Let $(X,d)$ be a metric space. If $F:I \rightrightarrows X$ is set-valued mappings which is compact and of bounded variation. If $t \in I$ and $x\in F(t),$ then if $F$ is a Lipschitzian then it admits a Lipschitizian selection $f$ such that Lipschitz constant of $f$ is bounded by the Lipschitz constant of $F.$
  \end{theorem}
  
\begin{theorem}
    Let set-valued mapping $F$ satisfies the condition of Theorem \ref{Bvar}. Then if $F$ is Lipschitz on $\mathcal{I}$ with respect to Hausdorff metric $H_d,$ then the mapping admits a Lipschitz selection.
\end{theorem}
\begin{proof}
    Under the assumptions of Theorem \ref{Bvar} along the same lines of proof, further if $H$ is Lipschitz on $\mathcal{I}$ with respect to Hausdorff metric $H_d.$ Hence, under the assumptions of Theorem \ref{moreselecton} of Belov and Chistyakov \cite{Belov} are applicable to the mapping $H$. Therefore, there exists a single-valued Lipschitz selection
    \[ h:\mathcal{I} \to \mathbb{R} \text{ such that } h(t) \in H(t)= _a^\rho\mathfrak{D}^\nu F(t) \text{ for every } t \in \mathcal{I}. \]
Further, the Lipschitz constant of $h$ is bounded above by the Lipschitz constant of $H.$ Thus, the set-valued Katugampola fractional integral mapping admits a Lipschitz continuous selection on $\mathcal{I}$, which proves the result.
\end{proof}
\begin{theorem}
  Let the set-valued mapping $F:\mathcal{I} \rightrightarrows \mathbb{R}$ be nonempty compact values with $\rho \neq -1, \nu > 0.$ Then for each $t \in \mathcal{I},$ let $H(t)= _a^\rho\mathfrak{D}^\nu F(t).$ By Theorem \ref{bounded}, for each $t \in \mathcal{I}$ the $H(t)$ is a non empty compact subset of $\mathbb{R}.$ Consider the extremal selections  
  \[h_+(t):=\sup H(t) \text{ and }  h_-(t):=\inf H(t)\]
  
  Let $F$ is non empty compact values, Borel measurable and integrably bounded on $\mathcal{I}$. Then the following holds: 
  \begin{enumerate}[label=(\alph*)]
      \item If $ \rho \neq -1, \nu > 0,$ then $H$ is continuous on $\mathcal{I}$ and the extreme selections $h_+$ and $h_-$ are continuous selections of $H$ on $\mathcal{I}$. 
      \item If $F$ is also convex valued and of bounded variation on $\mathcal{I}$ and $ \rho \neq -1, \nu > 1,$ then $H$ is bounded variation on $\mathcal{I}$ and the extreme selections $h_+$ and $h_-$ are of bounded variation on $\mathcal{I}.$ Also 
      \[ V_a^b(h_+) \le V_a^b(H) \text{ and } V_a^b(h_-) \le V_a^b(H).\]
      \item If $F$ is also convex valued and Lipschitz on $\mathcal{I}$ and $ \rho \neq -1, \nu > 1,$ then $H$ is Lipschitz on $\mathcal{I}$ and the extreme selections $h_+$ and $h_-$ are Lipschitz on $\mathcal{I}.$ Also Lipschitz constant of both $h_+$ and $h_-$ are bounded by the Lipschitz constant of $H.$
  \end{enumerate}
\end{theorem}
\begin{proof}
    By Theorem \ref{bounded}, for each $t \in \mathcal{I},$ the set $H(t)=_a^\rho\mathfrak{D}^\nu F(t)$ is a non-empty compact subset of $\mathbb{R}.$ Therefore infimum and supremum element of $H(t)$ exist for each $t \in \mathcal{I}.$ Hence the selections $h_-$ and $h_+$ are well defined.  
    \begin{enumerate}[label=(\alph*)]
        \item Under the assumptions given in Theorem, the definition of Hausdorff metric implies that 
        \[ |h_+(t)-h_+(s)| \le H_d(H(t),H(s)) \text{ for all } t,s \in \mathcal{I}\]
        and similarly, 
        \[ |h_-(t)-h_-(s)| \le H_d(H(t),H(s)) \text{ for all } t_1,t_2 \in \mathcal{I} .\]
        Since $H$ is continuous, both $h_+$ and $h_-$ are continuous. 
        \item Now, both $h_+$ and $h_-$ are selections of $H$ and are of bounded variation on $\mathcal{I}.$ Let $ \Gamma=\{a=t_0<t_1 <\ldots < t_n=b  $ be any partition of $\mathcal{I}.$ We obtain 
        \[ |h_+(t_i)-h_+(t_{i-1})| \le H_d(H(t_i),H(t_{i-1}))\]
        for each $i=1,2,\ldots,n.$ Summing over all $i,$ we get 
        \[ \sum_{i=1}^n |h_+(t_i)-h_+(t_{i-1})| \le \sum_{i=1}^nH_d(H(t_i),H(t_{i-1})) \]
        Taking the supremum over all partition yields, 
        \[ V_a^b(h_+)\le V_a^b(H).\]
        Repeating the same arguments for $h_-,$ we obtain that $V_a^b(h_-)\le V_a^b(H).$ Therefore both extremal selections are of bounded variation on $\mathcal{I}.$
         \item Assuming that the assumptions given in Theorem \ref{Bvar} are satisfied, then  $H$ is Lipschitz on $\mathcal{I}$ with Lipschitz constant $K.$ Hence for all $t,s \in \mathcal{I},$ we have 
         \[ |h_+(t)-h_+(s)| \le H_d(H(t),H(s)) \le K|t-s|.\]
         Similar inequality hols for the $h_-$ also. Thus both $h_+$ and $h_-$ are Lipschitz on $\mathcal{I}.$ Moreover, 
         Lipschitz constant of both $h_+$ and $h_-$ are bounded by the Lipschitz constant of $H.$ 
    \end{enumerate}
\end{proof}
\section{Conclusion and Future Work}
In this paper, we develop the theory of the Katugampola fractional integral for set-valued mappings, which extends the concept of Riemann Liouville fractional integral to the more broder Katugampola operator. Further, we studied the existence of selection for the  set-valued Katugampola fractional integral based on the selection principle of  Belov and Chistyakov. For further work, one can explore the set-valued Katugampola fractional integral on the rectangular domain of $[a,b]\times[c,d].$ For the single valued maps the dimensional properties and  the Hausdorff and Box dimensions of the graph of Katugampola fractional integral are extensively explored. An interesting question is whether an analogous dimension theory can be developed for the set-valued Katugampola fractional integral.
\section*{Statements and Declarations}
 The authors have no competing interests to declare that are relevant to the content of this article. There are no data associated with this paper. 
   
\section*{Acknowledgements}
 The first author is supported by the University Grant Commission in the form of Junior Research Fellow.

\bibliographystyle{amsalpha}

\end{document}